\documentclass[12pt]{amsart}
\usepackage{txfonts}
\usepackage{amscd,amssymb, amsmath,eucal}
\usepackage[colorlinks,plainpages,backref,urlcolor=blue]{hyperref}
\usepackage{verbatim}
\usepackage{xcolor}

\newtheorem{theorem}{Theorem}[section]
\newtheorem{corollary}[theorem]{Corollary}
\newtheorem{lemma}[theorem]{Lemma}
\newtheorem{prop}[theorem]{Proposition}

\theoremstyle{definition}

\newtheorem{example}[theorem]{Example}
\newtheorem{remark}[theorem]{Remark}

\newtheorem{question}[theorem]{Question}

\newcommand{\C}{\mathbb{C}}

\newcommand{\PP}{\mathbb{P}}

\DeclareMathAlphabet{\pazocal}{OMS}{zplm}{m}{n}
\newcommand{\A}{{\pazocal{A}}}

\newcommand{\CC}{{\mathcal{C}}}

\def\dot{\mathchar"013A}
\newcommand{\hdot}{{\raise1pt\hbox to0.35em{\Huge $\dot$}}}

\begin{document}
\date{\today}

\title[A generalized Saito criterion, from line arrangements to curves]%
{A generalized Saito criterion, from line arrangements to curves} 
\author{Anca~M\u acinic$^*$}
\thanks{$^*$ Partially supported by a grant of the Ministry of Research, Innovation and Digitization, CNCS - UEFISCDI, project number PN-IV-P1-PCE-2023-2001, within PNCDI IV}

\subjclass[2020]{14H50, 32S22 (Primary);   13N15,  13D02 (Secondary)}

\keywords{ plus-one generated curve; type 2 curve;  generalized Saito criterion}

\begin{abstract}

We formulate a generalized Saito criterion for complex projective plane curves, 
building on 
 a recent generalized Saito criterion for plus-one generated arrangements of hyperplanes due to Junyan Chu.
\end{abstract}

\maketitle

\begin{center}
{\it Dedicated to the memory of Mihai Tib\u{a}r}
\end{center}

\section{Introduction}

Since the introduction of the notion of plus-one generated (POG)  arrangements of hyperplanes by Abe in \cite{A:POG}, the study of POG arrangements of hyperplanes and of POG complex projective plane curves (the POG property was extended to curves in \cite{DS0}) is a very effervescent  area of research (see for instance \cite{ADen, ADP, AIM, Chu, Dim, DS1, M, MP1, MP2, MV}).
 The main  motivation comes from the connection with the freeness property, which is at the center of the emblematic Terao Conjecture.  The author proves in \cite{A:POG} that {\it next to} a free projective line arrangement $\A$ only free or POG arrangements can be found, the nearness to $\A$ being defined here by addition-deletion. More precisely, we have the following result.

\begin{theorem}(\cite{A:POG})
Let $\A$ be a free projective line arrangement. Then by adding or deleting a line one obtains either a free or a POG arrangement.
\end{theorem} 
 
The same phenomenon happens for (reduced complex plane projective) curves: near a free curve (i.e. by addition or deletion of an irreducible component which is a line) there can be only free or POG curves, see \cite{Dim, MP2}. Recently the notion of {\it type} of a curve was introduced in \cite{ADP}, see ~\S\ref{sec:prelim}  for the definition. With this, a free curve is precisely a type $0$ curve and a POG curve is precisely a type $1$ curve and the above mentioned results on curves can be rephrased as: {\it given a free curve, by the addition or deletion of a line, seen as an irreducible component, one necessarily obtains a curve of type at most $1$}. 

This type of behaviour extends to addition-deletion of smooth conics, as shown in a series of recent results from \cite{M}
that characterize the addition / deletion of a smooth conic to a free curve. In particular, the following holds.
 \begin{theorem}(\cite{M})
 Given a free curve, by the addition or deletion of a smooth conic, as an irreducible component, one necessarily obtains a curve of type at most $2$.
 \end{theorem} 

In  \cite{Chu} the author presents a generalization of Saito's criterion to strictly plus-one generated (SPOG) arrangements  (a version of this result is recalled in Theorem \ref{thm:Chu-thm_1.4}) and states a series of related conjectures / open questions.  %Among them, whether the...
We show that  Theorem \ref{thm:Chu-thm_1.4} also holds for curves, we state a similar result in Theorem \ref{thm:Chu_thm_consequence}, and we address some of these open questions, see Remark \ref{rem:counter} and Proposition \ref{thm:type_leq2_char}.
% Proposition \ref{thm:type_leq2_char},  gives a necessary and sufficient, albeit obvious, criterion for a curve to be a type at most $k$ curve.

\section*{Acknowledgments}
After sharing this note with Junyan Chu it was brought to our attention that in a submitted unpublished version of the paper \cite{Chu}
 \cite[Conjecture 4.2]{Chu} was removed, since the author herself found counterexamples to it,
and that this version contains as well as a construction of the ideal J(G), for arrangements of hyperplanes.
Also that there is work in progress by J. Chu and S. Kaji that could lead to a refinement of our Proposition \ref{thm:type_leq2_char}, by distinguishing types 2A and 2B in the case of type $2$ curves.
We thank Junyan Chu for looking over this note and for sharing her ongoing work and ideas, as well as for helpful remarks and suggestions that led to an improved exposition.

\section{Preliminaries}
\label{sec:prelim}
Throughout the paper we will only work with reduced curves in  the complex projective plane, which  we will call  simply curves. Relevant classes of examples, for the purpose of this note, are projective line arrangements and arrangements of conics and lines in $\PP^2\C$. Denote $S := \C[x,y,z]$ and $Der(S):=\{ \theta : S \rightarrow S \; |\; \theta(fg) = f\theta(g) + g\theta(f), \; \forall f,g \in S \}$, the module of derivations of $S$. Obviously, $Der(S)  = S \partial _x \oplus S \partial_y \oplus S \partial_z.$

Given a curve $\CC: f_{\CC}=0$, denote by $D(\CC)$ its associated module of  derivations, defined by 
$$
D(\CC):= \{\theta \in Der(S) \; | \;  \theta(f_{\CC}) \in f_{\CC}\cdot S\}.
$$

It is easy to see that  $\theta_E:= x \partial x + y \partial y +z \partial z \in D(\CC)$, for any curve $\CC$. Moreover, $D(\CC)$ admits a decomposition as  a direct sum
$D(\CC) = D_0(\CC) \oplus S \theta_E,$
where $$D_0(\CC) := \{ \theta \in  Der(S) \; | \; \theta(f_{\CC})=0 \}.$$
To make explicit the direct sum decomposition of $D(\CC)$, given $\theta \in D(\CC)$ with $ \theta(f_{\CC}) = g  f_{\CC}, \; g \in S$,
\begin{equation}
\label{eq:theta_0}
\theta^0 = \theta- \deg( f_{\CC})^{-1} g  \theta_E \in D_0(\CC).
\end{equation}

\begin{comment}
We recall that is well known that $D_0(\CC)$ is isomorphic to the graded S-module of Jacobian syzygyes of $f_\CC$.
 $AR(f_{f_{\CC}}):= \{(a,b,c) \in S^3 \; |\; a  \partial x(f_{\CC}) + b \partial y(f_{\CC}) + c \partial x(f_{\CC}) =0\}.$
 This isomorphism $D_0(\CC) \rightarrow AR(f_{f_{\CC}})$ maps $\theta = a \partial x + b \partial y +c \partial z, \; a,b,c \in S$ to $(a,b,c) \in S^3$.
\end{comment}
 
 A curve $\CC$ is called {\it free} if $D(\CC)$ (or, equivalently, $D_0(\CC)$)  is a free S-module. $\CC$ is  called {\it plus-one generated}  
  if $D(\CC)$ admits a minimal system of generators $\{\theta_1, \theta_2, \theta_3, \theta_4\}$  satisfying a unique relation 
  \begin{equation}
\label{eq:POG_rel}
  f_1 \theta_1 + f_2 \theta_2 + f_3 \theta_3  + f_4 \theta_4=0,
   \end{equation}
  where $f_i \in S, i = \overline{1,4}$ with $\deg(f_4) =1$. We recall that all POG curves are SPOG (this is defined as POG plus the condition $f_4 \neq 0$), since all the coefficients $f_{\geq 2}$ of the unique POG relation are non-zero, see \cite[Corollary 2.2]{DS0}.
\smallskip

We introduce now the general set-up that enables us to define the type of a curve.
Let $G = \{\theta_1, \theta_2, \dots, \theta_r\}$ be a minimal system of generators for $D(\CC)$ of ordered degrees $1=d_1 \leq d_2 \leq  \dots \leq  d_r$, to which we will refer as the {\it generalized exponents} of $\CC$. 
Denote $d:= \deg(f_{\CC})$. The {\it type} of the curve $\CC$ is defined in \cite{ADP} 
as the positive integer $$t(\CC):=d_2+d_3-d+1.$$

We recall a result from \cite{DS0} that justifies the above definition.
 
\begin{theorem}
\label{thm: gen_exp}
\begin{enumerate}
\item $\CC$ is free if and only if  $t(\CC) =0$.  %$d_1 + d_2 = d-1$
\item $\CC$ is plus-one generated if and only if $t(\CC)=1$.  %$d_1 + d_2 = d$
\end{enumerate}
\end{theorem}

\subsection{Saito type criteria}
For $\{\theta_1, \theta_2, \theta_3\} \subset D(\CC)$ homogeneous derivations, define the $3 \times 3$  matrix $M(\theta_1, \theta_2, \theta_3)$ as the matrix whose line $i$ is consists of the coefficients in $S$ of the derivation $\theta_i$. The below Proposition \ref{prop:Delta|f} is most likely known, but we give here a proof for completion. The proof is  based on the following result from  \cite{dPW}.

\begin{lemma}(\cite[Lemma 3.1]{dPW})
\label{lem:3.1}
Let $\{ \theta^0_1, \theta^0_2\}  \subset D_0(\CC)$ homogeneous derivations, $$\theta^0_1 = a_1 \partial x + b_1 \partial y +c_1 \partial z, \; \theta^0_2 = a_2 \partial x + b_2 \partial y +c_2 \partial z.$$ Then there exists $h \in S$ such that  $$(b_1 c_2 - b_2 c_1, c_1 a_2 - c_2 a_1,  a_1 b_2 -a_2 b_1) = h(\frac{\partial f_{\CC}}{\partial x}, \frac{\partial f_{\CC}}{\partial y}, \frac{\partial f_{\CC}}{\partial z}).$$
\end{lemma}

\begin{corollary}
\label{cor:det}
\begin{enumerate}
\item For $\{\theta^0_1, \theta^0_2, \theta^0_3 \} \subset D_0(\CC)$ homogeneous derivations,  $$\det M(\theta^0_1, \theta^0_2, \theta^0_3)=0.$$
\item For $\{ \theta^0_1, \theta^0_2\}  \subset D_0(\CC)$ homogeneous derivations, there exists $g \in S$ such that 
$$\det M(\theta_E, \theta^0_1, \theta^0_2) = g f_{\CC}.$$
\end{enumerate}
\end{corollary}

\begin{proof}
Immediate from Lemma \ref{lem:3.1}.
\end{proof}

\begin{prop}
\label{prop:Delta|f}
For any  set of homogeneous derivations $\{ \theta_1, \theta_2, \theta_3\} \subset D(\CC)$, there exists $g \in S$ such that  
$$
\det M(\theta_1, \theta_2, \theta_3) = gf_{\CC}.
$$
\end{prop}

\begin{proof}
Since $\theta_i \in D(\CC)$, there exists $g_i \in S$ such that $\theta_i(f_{\CC}) = g_i f_{\CC}, \; i=\overline{1,3}$.
Write, as in \eqref{eq:theta_0}, $\theta_i = \theta^0_i +d^{-1} g_i \theta_E$. Then $M(\theta_1, \theta_2, \theta_3)  = M(\theta^0_1 + d^{-1} g _1 \theta_E, \theta^0_2 + d^{-1} g _2 \theta_E, \theta^0_3 + d^{-1} g _3 \theta_E),$ so, by basic properties of determinants,
$$\det M(\theta_1, \theta_2, \theta_3) = \det M(\theta^0_1, \theta^0_2, \theta^0_3) + \sum_{i<j} \pm  d^{-2}g_ig_j\det M(\theta_E, \theta^0_i, \theta^0_j)$$ 
 The claim now follows from Corollary \ref{cor:det}.
\end{proof}

For $\{\theta_i, \theta_j, \theta_k \} \subset D(\CC)$ homogeneous derivations, define $g_{ijk} \in S$ by the condition 
\begin{equation}
\label{eq:prop_div}
\det  M(\theta_i, \theta_j, \theta_k) = g_{ijk} f_{\CC}.
\end{equation} 
The existence of  $g_{ijk}$ is ensured by Proposition \ref{prop:Delta|f}.

The next result was originally stated in  \cite[Lemma 3.1 \& Prop. 3.2]{Chu} for arrangements of hyperplanes. This result holds for curves, with  the same proof as in \cite{Chu}, keeping in mind the divisibility property for curves stated in Proposition \ref{prop:Delta|f}.
 
 \begin{prop}
 \label{prop:lemmas_extended_to_curves}
 Let  $\{\theta_1, \cdots ,\theta_4 \} \subset D(\CC)$ be homogeneous derivations and denote $g_i:=(-1)^ig_{1 \cdots \hat{i} \cdots 4}.$
 Then 
 \begin{enumerate}
 \item $\sum_{i=1}^4 g_i \theta_i =0$.
 \item If $g_i \neq 0$, then, for any $\theta \in D(\CC)$, $g_i \theta \in \langle \theta_1, \dots \hat{\theta_i}, \cdots , \theta_4 \rangle$.
 \item If there is a relation of type $\sum_{j \neq i} f_j \theta_j + g_i \theta =0$, then the coefficients $f_j \in S$ are uniquely determined by the identity %from \eqref{eq:prop_div} 
 $(-1)^j \det M( \theta_1, \cdots \hat{\theta_j}, \cdots,  \theta_4, \theta)=f_j f_{\CC}.$
 \end{enumerate}
  \end{prop}
 
\section{Main Results}
\label{sec:main}

\subsection{Generalized Saito criterion for curves}

In \cite{Chu} the author gives a Saito-type characterization of the POG property for arrangements, in particular obtaining the following result for arrangements of projective lines.

\begin{theorem}(\cite{Chu})
\label{thm:Chu-thm_1.4}
Let $\A$ be an arrangement of projective lines. Then $\A$ is SPOG  and the set $\{\theta_1, \cdots, \theta_4\}$ is a minimal generating system for $D(\A)$ that satisfies the minimal degree relation $\sum_{i=1}^4  g_i \theta_i =0$ if and only if $0 \neq g_4 \in S_1$ and $g_1, g_2, g_3 \in S_{>0}$ have no non-trivial divisor modulo $g_4$.
\end{theorem}

The proof  from \cite{Chu} of Theorem \ref{thm:Chu-thm_1.4}  works verbatim for plane projective curves as well, since it essentially uses a divisibility property as the one proved in Proposition \ref{prop:Delta|f}, alongside linear algebra.  So the above result can be restated as is, but for complex projective plane curves. \\

Actually, we can state and prove a variation of the above result that does not involve (directly) a fourth derivation and holds also for curves.

\begin{theorem}
\label{thm:Chu_thm_consequence}
Let $\CC$ be a non-free curve.  Then $\CC$ is POG if and only if there exists a set of homogeneous derivations $\{\theta_1, \theta_2, \theta_3 \} \subset D(\CC)$  such that $0 \neq g_{123} \in S_1$. 
\end{theorem}

\begin{proof}

Let us prove first the direct implication. If $\CC$ is POG with generalized  exponents $(1,d_2,d_3, d_4)$ then choose from a minimal system of generators for $D(\CC)$ three homogeneous derivations $\theta_1, \theta_2, \theta_3$, such that $\deg(\theta_1) =1$, $\deg(\theta_2) =d_2$ and $\deg(\theta_3) =d_3$. 
 Recall that, since $\CC$ is POG, we have the identity $d_2+d_3 = \deg(\CC)$. Then $\deg(g_{123}) = 1+d_2+d_3 - \deg(\CC) = 1$. Moreover, since $\theta_1, \theta_2, \theta_3$ are S-independent, $g_{123} \neq 0$.
 \smallskip
 
 Conversely, assume that $\CC$ is non-free and such that there exists a set of homogeneous derivations $\{\theta_1, \theta_2, \theta_3\} \subset D(\CC)$ with $0 \neq g_{123} \in S_1$.  The condition $0 \neq g_{123}$ implies that the derivations $\theta_1, \theta_2, \theta_3$ are $S$-independent and the condition $g_{123} \in S_1$ implies that $\deg(\theta_1) + \deg(\theta_2) + \deg(\theta_3) = d+1$.  Keeping in mind that $\theta_1, \theta_2, \theta_3$ are $S$-independent, if $1 \leq d_2 \leq \cdots \leq d_r$ denote the generalized exponents of $\CC$, then the following inequality holds  
 $$\deg(\theta_1) + \deg(\theta_2) + \deg(\theta_3) \geq 1+d_2+d_3,$$
  hence $d+1 \geq 1+d_2+d_3$. Since $\CC$ is not free, this implies, by Theorem \ref{thm: gen_exp}, that $\CC$ must necessarily be POG.
 \end{proof}

\subsection{Bourbaki ideals and type $k$ curves} 
\label{ss:bourbaki}

For an arbitrary curve $\CC$ with  generalized exponents $1=d_1 \leq d_2 \leq \cdots \leq d_r$, a minimal resolution of $D(\CC)$ is of type:
\begin{equation}
\label{eq:res_arb_curve}
0 \rightarrow  \bigoplus_{i=2}^{r-2} S[-e_i] \rightarrow \bigoplus_{i=1}^{r} S[-d_i] \rightarrow D(\CC) \rightarrow 0
\end{equation}
where $e_2 \leq \cdots \leq e_{r-2}$.
By  \cite[Thm 0.2]{BM},  $$d_2 +d_3 = d-1+ \sum_{i=2}^{r-2} e_i - \sum_{i=2}^{r-2} d_{i+2} = d+1 + [\sum_{i=2}^{r-2}( e_i - d_{i+2})-2],$$
 and, by   \cite{HS}, %see also \cite{DS0},
  $e_i - d_{i+2} \geq 1, \; \forall i \in \overline{2,r-2}$.
  
 It follows  that a type $k$ curve is a curve such that  $ \sum_{i=2}^{r-2} (e_i - d_{i+2}) = k$. In particular, a type 2 curve  is  precisely a curve that satisfies one of the two conditions: either there exists a unique $i$ such that $ e_i - d_{i+2} = 2$ or there exists a pair $i \neq j$ such that
$ e_i - d_{i+2} = 1$ and $ e_j- d_{j+2} = 1$.% In the first case, $\CC$ is a three syzygy curve and in the second case, $\CC$ is a four syzygy curve.
 To summarize, in the notations from \eqref{eq:res_arb_curve}, if $\CC$ is a type $2$ curve, then either $r=4$ and $e_2 = d_4+2$, in which case $\CC$ is called of {\it type $2A$}, or $r=5$ and $e_2 = d_4 +1, \; e_3 = d_5+1$, in which case $\CC$ is called of {\it type $2B$}, by the convention established in \cite[Prop 1.12]{ADP}.

\begin{remark}
\label{rem:counter}
 \cite[Conjecture 4.2]{Chu}, which states  that, for a projective line arrangement $\A$, a resolution of $D(\A)$ is necessarily of the following type
$$0 \rightarrow  \bigoplus_{i=2}^{r-2} S[-d_{i+2}-1] \rightarrow  \bigoplus_{i=1}^{r} S[-d_i] \rightarrow D(\CC) \rightarrow 0 $$ 
does not hold for curves.  It also does not hold for arrangements of projective lines (equivalently, central arrangements of planes in $\C^3$). Any type $2A$ curve or projective line arrangement is a counterexample.  See for instance \cite[Theorems 3.1, 3.3 and Remarks 3.3, 3.7]{M} for a recipe to construct such counterexamples, by addition-deletion of a smooth conic to a free curve.
\end{remark}

\begin{example}
%Some yype $2A/2B$ curves, obtained by addition-deletion as in \cite{M}.
%The curve 
$\CC:  (x^2 + 2xy + y^2 + xz)(x^2 + xz + yz)(x^2 + xy + z^2)(x+y-z)y(x+z)(2x+y)(x^2 + 2xy - xz + yz)=0$ is a type $2A$ curve with generalized exponents $(1,6,7,7)$, which can be obtained by deletion of a smooth conic from a free curve, see  \cite[Example 4.4]{M}.
\end{example}

\subsubsection{The ideal $J(G)$}
\label{ss:J(G)}
Going back to the case of an arbitrary curve $\CC$, let $G = \{\theta_1 = \theta_E, \theta_2, \cdots, \theta_r\}$ be a minimal system of generators for $D(\CC)$, with degrees $1=d_1 \leq d_2 \leq \cdots \leq d_r$.  Denote by $\theta^0_i$ the derivation in $D_0(\CC)$ corresponding to $\theta_i$ via \eqref{eq:theta_0}  and by $G^0 = \{\theta_1 = \theta_E, \theta^0_2, \cdots, \theta^0_r\}$. Obviously, $G^0$ is also a minimal system of generators for $D(\CC)$.   We recall that $ d_2$, the minimal degree of a derivation in $D_0(\CC)$, is denoted by $mdr(f_{\CC})$.

The definition of type $2$ curves can be formulated in terms of an associated Bourbaki ideal, see for instance \cite{DS0, ADP} for details. Briefly, given a homogeneous derivation $\theta^0 \in D_0(\CC)$ of degree $mdr(f_{\CC})$, one defines the Bourbaki ideal associated to $\theta^0$,  denoted $B(\CC, \theta^0)$, as the homogeneous ideal in $S$ 
$$B(\CC, \theta^0): =\langle  g_{\phi} \;| \; \det M(\theta_E, \theta^0, \phi) = g_{\phi} f_{\CC}, \;  \phi \in D_0(\CC) \rangle$$
Corollary \ref{cor:det} ensures us that the Bourbaki ideal  is well defined, meaning that, for any $\phi \in D_0(\CC)$, there exists $g_{\phi} \in S$ such that
$\det M(\theta_E, \theta^0, \phi) = g_{\phi} f_{\CC}$. 
Given an arbitrary homogeneous ideal $I \subset S$, denote by $Indeg(I)$ the minimal degree of a non-zero element of $I$, called {\it initial degree} of $I$. With this definition, the type of the curve $\CC$ coincides to the initial degree of $B(\CC, \theta^0)$.

With the notations from the beginning of  ~\S\ref{ss:J(G)},  let $B(\CC, \theta^0_2)$ be the Bourbaki ideal of $\CC$ associated to $\theta^0_2$ and denote 
$$J(G):=\langle g_{ijk}| \; i \neq j  \neq k \in \overline{1,r} \rangle.$$
Furthermore, denote  $g_{ij}:= g_{1ij}$. So $B(\CC, \theta^0) = \langle g_{2k} |\;  k \in  \overline{1,r} \rangle \subset J(G)$.

\begin{lemma}
\label{lem:indeg}
\begin{enumerate}
\item  $Indeg J(G)  = t(\CC)$
\item  $J(G) = J(G^0)$
\end{enumerate}
\end{lemma}

\begin{proof}
(1) This is obvious, since $t(\CC) = Indeg B(\CC, \theta^0_2) = 1+d_2+d_3-d$ and $G$ is a minimal generating set for $D(\CC)$, where the degree $d$ curve $\CC$ has generalized exponents  $1 \leq d_2 \leq \cdots \leq d_r$.

(2) It is enough to notice that $\det M(\theta_1, \theta_j, \theta_k) = \det M(\theta_1, \theta^0_j, \theta^0_k)$ and $g_{ijk} \in \langle g_{ij}, g_{jk}, g_{ik} \rangle$, for any $i,j,k \geq 2$.
\end{proof}

\begin{remark}
In general, if $G$ is a system of generators for $D(\CC)$,  $B(\CC, \theta^0_2) \subsetneq J(G)$.
\end{remark}

In the above statements %, Lemma \ref{lem:indeg} and Corollary \ref{cor:indeg2},  
$G$ was assumed to be a system of generators for $D(\CC).$
But even for an arbitrary subset of homogeneous derivations $H \subset D(\CC)$  we can define an associated ideal $J(H)$ analogously, by  $J(H): =\langle g_{ijk}| \{\theta_i, \theta_j, \theta_k\} \subset H \rangle$. In this case, if $J(H) \neq 0$, obviously $Indeg(J(H)) \geq t(\CC)$ and we immediately have the following.

\begin{prop}
\label{thm:type_leq2_char}
 $\CC$ is of type at most $k$ if and only if there exists $H \subset D(\CC)$ such that $Indeg(J(H))=k$.
\end{prop}

\begin{remark}
Let $H \subset D(\CC)$ be such that $Indeg(J(H))=2$. If $\CC$ is free or POG, then $H$ cannot be a system of generators for $D(\CC)$.
\end{remark}

\begin{lemma}
\label{lem:non_rel_thetaE23}
If $\{\theta_E, \theta_2, \theta_3 , \cdots, \theta_r\}$ is a minimal generating system of $ D(\CC)$ such that $1 \leq \deg(\theta_2) \leq \deg(\theta_3) \leq \cdots \leq \deg(\theta_r)$, then there is no non-trivial relation that involves only $\{\theta_E, \theta_2, \theta_3\}$.
\end{lemma}

\begin{proof}
Assume to the contrary that there is a relation $\alpha_E \theta_E + \alpha_2 \theta_2 + \alpha_3 \theta_3 =0$, $\alpha_E, \alpha_2, \alpha_3 \in S$. Since $D(\CC) = D_0(\CC) \oplus S \theta_E$ , the relation implies that $\alpha_2 \theta^0_2 + \alpha_3 \theta^0_3 =0$, where $\theta_i := \theta^0_i + \deg( f_{\CC})^{-1} g  \theta_E, \; i=2,3$, see \eqref{eq:theta_0}. Let $\alpha = \gcd(\alpha_2,  \alpha_3)$. Then there exist $\alpha'_2,  \alpha'_3 \in S, \; \gcd(\alpha'_2,  \alpha'_3) = 1$ and $\alpha_2 = \alpha \alpha'_2, \alpha_3 = \alpha \alpha'_3$. Hence $\alpha'_3$ divides $\theta^0_2 $, i.e. there exists $\theta^{01}_2 \in D_0(\CC)$ such that  $\theta^0_2 = \alpha'_3 \theta^{01}_2$. If $\deg(\alpha'_3) > 0$ that would contradict the minimality of the degree of $\theta^0_2$, which is supposed to be $mdr(f_{\CC})$. So necessarily $\alpha'_3 \in \C^*$. But then $ \theta^0_2, \theta^0_3$ are S-dependent, which contradicts the minimality of the generating system. In conclusion a relation of type $\alpha_E \theta_E + \alpha_2 \theta_2 + \alpha_3 \theta_3 =0$ must have zero coefficients.\\
\end{proof}

\begin{remark}
\label{rem:rel_2B}
If  $\CC$ is a type $2B$ curve, say, with minimal generating system for  $D(\CC)$ consisting of the degree ordered derivations $\{\theta_E, \theta_2, \theta_3, \theta_4, \theta_5\}$ then, by Lemma \ref{lem:non_rel_thetaE23}, the two generating syzygies of degrees $e_2,e_3$ (see \eqref{eq:res_arb_curve}) necessarily involve terms with non-trivial coefficients for either  $\theta_4$ or  $\theta_5$.  
Consider the  \ref{prop:lemmas_extended_to_curves}(1) type relations for $\{\theta_E, \theta_2, \theta_3, \theta_4\}$, respectively $\{\theta_E, \theta_2, \theta_3, \theta_5\}$.  The (degree $2$) coefficients of $\theta_4, \theta_5$ in these relations are equal, up to $\C^*$, to $g_{123}$.
\end{remark}

\begin{question}
It would be interesting to have a characterization of type $2$ curves along the lines of Theorem \ref{thm:Chu_thm_consequence}. The above computations from \ref{lem:non_rel_thetaE23} and \ref{rem:rel_2B} might be of relevance for this endeavour.
\end{question}

\bigskip

Anca~M\u acinic,
Simion Stoilow Institute of Mathematics of the Romanian Academy, 
P.O. Box 1-764, RO-014700 Bucharest, Romania. \\
\nopagebreak
\textit{E-mail address:} \texttt{anca.macinic@imar.ro}


\begin{thebibliography}{00}
 
\bibitem{A:POG} T.~Abe,
{\em Plus-one generated and next to free arrangements of hyperplanes},
Int. Math. Res. Not.  (2021),  Vol. {\bf 2021}, Issue 12,  9233--9261.

%\bibitem{AD} T.~Abe, A.~Dimca, {\em Splitting types of bundles of logarithmic vector fields along plane curves}, Internat. J. Math. {\bf 29(8)} (2018),  Art. Id. 1850055.

\bibitem{ADen} T. ~Abe,  G. ~Denham,
{\em  Deletion-Restriction for Logarithmic Forms on Multiarrangements}, 
Advances in Applied Mathematics
Article: 103114, Volume 179 (2026).

\bibitem{ADP} T.~Abe, A.~Dimca, P.~Pokora,
{\em A new hierarchy for complex plane curves}, 
arXiv:2410.11479

\bibitem{AIM}  T.~Abe, D.~Ibadula, A.~Macinic, {\em On some freeness-type properties for line arrangements},  Ann. Sc. Norm. Super. Pisa Cl. Sci. (5) Vol. XXV (2024), 427--447.
   
 \bibitem{Chu} J.~Chu,
 {\em Generalizing Saito’s Criterion for Nonfree Arrangements}, arXiv:2603.21101 (2026)
 
\bibitem{Dim} A.~Dimca,
{\em On plus-one generated curves arising from free curves},
Bulletin of Mathematical Sciences, https://doi.org/10.1142/S1664360724500073.

\bibitem{DS0} A.~ Dimca, G.~ Sticlaru,
{\em Plane curves with three syzygies, minimal Tjurina curves curves, and nearly cuspidal curves},
 Geom. Dedicata  {\bf 207} (2020), 29--49.
 
\bibitem{DS1}  A.~ Dimca, G.~ Sticlaru,
{\em Plus-One Generated Curves, Briançon-Type Polynomials and Eigenscheme Ideals},
 Results Math. 80, 51 (2025).
 
\bibitem{HS} S. H. Hassanzadeh, A. Simis, 
{\em Plane Cremona maps: Saturation and regularity of the base ideal}, 
J. Algebra 371 (2012), 620--652.
 
  \bibitem{M} A. ~M\u{a}cinic, 
  {\em Addition-deletion of a smooth conic for free curves}, 
  Nagoya Mathematical Journal, Volume 261 (2026) e20
 
 \bibitem{MP1} A. ~M\u{a}cinic, P.~Pokora, {\em On plus-one generated conic-line arrangements with simple singularities} arXiv: 2309.15228,  https://arxiv.org/abs/2309.15228, to appear in Rendiconti Lincei – Matematica e Applicazioni
 
 \bibitem{MP2}  A. ~M\u{a}cinic, P.~ Pokora, {\em Addition–deletion results for plus-one generated curves}, J Algebr. Comb. Vol. 60, 723 -- 734 (2024). https://doi.org/10.1007/s10801-024-01350-x

 \bibitem{MV} A. ~M\u{a}cinic, J. ~Vall\`es,    {\em A geometric perspective on plus-one generated arrangements of lines},   Int. J. Math.Vol. 35, No. 10, 2450034  (2024)
  
  \bibitem{BM} M. A. Marco-Buzunariz and J. Martıin-Morales, 
  {\em Graded Betti numbers of the logarithmic derivation module},
   Comm. Algebra 38.11 (2010), 4348--4361.
  
 \bibitem{dPW} A.A. du Plessis, C.T.C. Wall,   {\em Application of the theory of the discriminant to highly singular plane curves}, Math. Proc. Camb. Phil. Soc., 126 (1999), 259-266.

% \bibitem{Terao} H. ~Terao,  {\em Arrangements of hyperplanes and their freeness   I, II},  J. Fac. Sci. Univ. Tokyo 27 (1980), 293--320.
  
\end{thebibliography}
\end{document}